\documentclass[12pt]{article}

\usepackage{amsmath}
\usepackage{mathtools}
\usepackage{amssymb}
\usepackage{enumerate}
\usepackage{here}
\usepackage{color}
\usepackage{mathrsfs}
\usepackage{multirow}
\usepackage{fullpage}
\usepackage{epsfig,psfrag}

\usepackage{cite}

\usepackage{dsfont}
\definecolor{lblue}{RGB}{0,110,152}

\definecolor{dred}{RGB}{171,67,53}

\providecommand{\red}[1]{\color{red}{#1}\color{black}\hspace{0pt}}

\newtheorem{theorem}{Theorem}

\newtheorem{proposition}[theorem]{Proposition}

\newtheorem{remark}[theorem]{Remark}

\newcommand{\mendth}{\hfill \ensuremath{\vartriangle}}

\providecommand{\red}[1]{\textcolor[rgb]{0.98,0.00,0.00}{#1}}

\DeclareMathOperator{\He}{Sym}

\DeclareMathOperator*{\dom}{dom}

\DeclareMathOperator{\eps}{\varepsilon}

\newenvironment{proof}{{\it Proof :~}}{\hfill$\diamondsuit$\\}

\begin{document}



\title{Stability analysis of LPV systems with piecewise differentiable parameters}

\author{Corentin Briat and Mustafa Khammash\thanks{Corentin Briat and Mustafa Khammash are with the Department of Biosystems Science and Engineering, ETH-Z\"{u}rich, Switzerland; email: mustafa.khammash@bsse.ethz.ch, corentin.briat@bsse.ethz.ch, corentin@briat.info; url: https://www.bsse.ethz.ch/ctsb/, http://www.briat.info.}}

\date{}

\maketitle


\begin{abstract}
Linear Parameter-Varying (LPV) systems with piecewise differentiable parameters is a class of LPV systems for which no proper analysis conditions have been obtained so far. To fill this gap, we propose an approach based on the theory of hybrid systems. The underlying idea is to reformulate the considered LPV system as an equivalent hybrid system that will incorporate, through a suitable state augmentation, information on both the dynamics of the state of the system and the considered class of parameter trajectories. Then, using a result pertaining on the stability of hybrid systems, two stability conditions are established and shown to naturally generalize and unify the well-known quadratic and robust stability criteria together. The obtained conditions being infinite-dimensional, a relaxation approach based on sum of squares programming is used in order to obtain tractable finite-dimensional conditions. The approach is finally illustrated on two examples from the literature.
\end{abstract}

\section{Introduction}

Linear Parameter-Varying (LPV) systems \cite{Toth:10,Mohammadpour:12,Briat:book1} are an important class of linear systems that can be used to model linear systems that intrinsically depend on parameters \cite{Wu:95} or to approximate nonlinear systems in the objective of designing gain-scheduled controllers \cite{Shamma:88phd,Shamma:92}; see \cite{Briat:book1} for a classification attempt. They have been, since then, successfully applied to a wide variety of real-world systems such as automotive suspensions systems \cite{Poussot:08b,Poussot:10}, robotics \cite{KajiwaraAG:99a},  aperiodic sampled-data systems \cite{Robert:10}, aircrafts  \cite{Gilbert:10, Seiler:12,Pfifer:15}, etc. The field has also been enriched with very broad theoretical results and numerical tools \cite{Packard:94a,Apkarian:95,Apkarian:95a,Apkarian:98a,Wu:01,Wu:06b,Scherer:01,Bokor:05a,Scherer:12,Scherer:15,Briat:book1,Briat:15d,Balas:15,Peni:16,Wang:16}.


The objective of the paper is to extend the stability results in \cite{Briat:15d} for LPV systems with piecewise constant parameters to the case of piecewise differentiable parameters. Such parameter trajectories may arise whenever an (impulsive) LPV system is used to approximate a nonlinear impulsive system or in linear systems where parameters naturally have such a behavior. Finally, they can also be used to approximate parameter trajectories that exhibit intermittent very fast, yet differentiable, variations. Despite being similar to the case of piecewise constant parameters, the fact that the parameters are time-varying between discontinuities leads to additional difficulties which prevent from straightforwardly extending the approach developed in \cite{Briat:15d}. To overcome this, we propose a different, more general, approach based on hybrid systems theory \cite{Goebel:12}. Firstly, we equivalently reformulate the considered LPV system into a hybrid system that captures in its formulation both the dynamics of the state of the system and the considered class of parameter trajectories. Recent results from hybrid systems theory \cite{Goebel:12} combined with the use of quadratic Lyapunov functions are then applied in order to derive sufficient stability conditions for the considered hybrid systems and, hence, for the associated  LPV system and class of parameter trajectories. We then prove that the obtained stability conditions generalize and unify the well-known quadratic stability and robust stability conditions, which can be recovered as extremal/particular cases. The obtained infinite-dimensional stability conditions are reminiscent of those obtained in \cite{Briat:13d,Briat:14f,Briat:15i}, a fact that strongly confirms the connection between the approaches. To make these conditions computationally verifiable, we propose a relaxation approach based on sum of squares \cite{Putinar:93,Parrilo:00,sostools3} and semidefinite programming \cite{Sturm:01a}. The approach is finally validated through two examples.

\noindent \textbf{Outline.} The structure of the paper is as follows: in Section \ref{sec:preliminary} preliminary definitions and results are given. Section \ref{sec:stab} develops the main results of the paper. Examples are finally given in Section \ref{sec:examples}.

\noindent \textbf{Notations.} The set of nonnegative integers is denoted by $\mathbb{N}_0$. The set of symmetric matrices of dimension $n$ is denoted by  $\mathbb{S}^n$ while the cone of positive (semi)definite matrices of dimension $n$ is denoted by ($\mathbb{S}^n_{\succeq0}$) $\mathbb{S}_{\succ0}^n$. For some $A,B\in\mathbb{S}^n$, the notation that $A\succ(\succeq)B$ means that $A-B$ is positive (semi)definite. The maximum and the minimum eigenvalue of a symmetric matrix $A$ are denoted by $\lambda_{max}(A)$ and $\lambda_{min}(A)$, respectively. A function $\alpha:\mathbb{R}_{\ge0}\mapsto\mathbb{R}_{\ge0}$ is of class $\mathcal{K}_\infty$ (i.e. $\alpha\in\mathcal{K}_\infty$), if it is continuous, zero at zero, increasing and unbounded. A function $\beta:\mathbb{R}_{\ge0}\mapsto\mathbb{R}_{\ge0}$ is positive definite if $\beta(s)>0$ for all $s>0$ and $\beta(0)=0$. Given a vector $v\in\mathbb{R}^d$ and a closed set $\mathcal{A}\subset\mathbb{R}^d$, the distance of $v$ to the set $\mathcal{A}$ is denoted by $|v|_\mathcal{A}$ and is defined by $|v|_\mathcal{A}:=\inf_{y\in\mathcal{A}}|v-y|$. For any differentiable function $f(x,y)$, the partial derivatives with respect to the first and second argument evaluated at $(x,y)=(x^*,y^*)$ are denoted by $\partial_x f(x^*,y^*)$ and $\partial_y f(x^*,y^*)$, respectively.

\section{Preliminaries}\label{sec:preliminary}

\subsection{LPV systems}

LPV systems are dynamical systems that can be described as
\begin{equation}\label{eq:mainsyst}
\begin{array}{rcl}
    \dot{x}(t)&=&A(\rho(t))x(t)\\
    x(0)&=&x_0
\end{array}
\end{equation}
where $x,x_0\in\mathbb{R}^n$ are the state of the system and the initial condition, respectively. The matrix-valued function $A(\cdot)\in\mathbb{R}^{n\times n}$ is assumed to be bounded and continuous. The parameter vector trajectory $\rho:\mathbb{R}_{\ge0}\to\mathcal{P}\subset\mathbb{R}^N$, $\mathcal{P}$ compact and connected, is assumed to be piecewise differentiable with derivative in $\mathcal{D}\subset\mathbb{R}^N$, where $\mathcal{D}$ is also compact and connected. When the parameters are independent of each other, then we can assume that $\mathcal{P}$ is a box and that the following decompositions hold $\mathcal{P}=:\mathcal{P}_1\times\ldots\times\mathcal{P}_N$ where $\mathcal{P}_i:=[\underline{\rho}_i,\ \bar{\rho}_i]$, $\underline{\rho}_i\le\bar{\rho}_i$ and $\mathcal{D}=:\mathcal{D}_1\times\ldots\times\mathcal{D}_N$ where $\mathcal{D}_i:=[\underline{\nu}_i,\ \bar{\nu}_i]$, $\underline{\nu}_i\le\bar{\nu}_i$. We also define the set of vertices of $\mathcal{D}$ as $\mathcal{D}^v$; i.e. $\mathcal{D}^v:=\{\underline{\nu}_1,\ \bar{\nu}_1\}\times\ldots\{\underline{\nu}_N,\ \bar{\nu}_N\}$. 

\subsection{Hybrid systems}

Let us consider here the following hybrid system
\begin{equation}\label{eq:hybrid}
\begin{array}{rcl}
  \dot{\chi}(t)&\in&F(\chi(t))\ \textnormal{if }\chi(t)\in C\\
  \chi(t^+)&\in&G(\chi(t))\ \textnormal{if }\chi(t)\in D
  \end{array}
\end{equation}
where $\chi(t)\in\mathbb{R}^d$, $C\subset\mathbb{R}^d$ is open, $D\subset\mathbb{R}^d$ is compact and $C\subset G(D)$. The flow map and the jump map are the set-valued maps $F:C\rightrightarrows\mathbb{R}^n$ and $G:D\rightrightarrows C$, respectively. Note that the trajectories of the above system are left-continuous and the right-handed limit is given denoted by $\chi(t^+)=\lim_{s\downarrow t}\chi(s)$. It is often convenient to define a solution (or hybrid arc) $\phi$ to the above system over a hybrid time domain $\dom\phi\subset\mathbb{R}_{\ge0}\times\mathbb{Z}_{\ge0}$ where the first component denotes the usual continuous-time while the second one counts the number of jumps\footnote{See \cite{Goebel:12} for more details about the solutions of \eqref{eq:hybrid}.}. We also assume for simplicity that the solutions are complete (i.e. $\dom\phi$ is unbounded). We then have the following stability result:
 \begin{theorem}[\cite{Goebel:12}]\label{th:goebel}
   Let $\mathcal{A}\subset\mathbb{R}^d$ be closed. Assume that there exist a function $V:\bar{C}\cup D\mapsto\mathbb{R}$ that  is continuously differentiable on an open set containing $\bar C$ (i.e. the closure of $C$), functions $\alpha_1,\alpha_2\in\mathcal{K}_\infty$ and a continuous positive definite function $\alpha_3$ such that
   \begin{enumerate}[(a)]
     \item $\alpha_1(|\chi|_\mathcal{A})\le V(x)\le\alpha_2(|\chi|_\mathcal{A})$ for all $\chi\in\bar{C}\cup D$;
     \item $\langle\nabla V(\chi),f\rangle\le-\alpha_3(|\chi|_\mathcal{A})$ for all $\chi\in C$ and $f\in F(\chi)$;
    \item  $V(g)-V(\chi)\le0$ for all $\chi\in D$ and $g\in G(\chi)$.
   \end{enumerate}
    Assume further that for each $r>0$, there exists a $\gamma_r\in\mathcal{K}_\infty$ and an $N_r\ge0$ such that for every solution $\phi$ to the system \eqref{eq:hybrid}, we have that $|\phi(0,0)|_{\mathcal{A}}\in(0,r]$, $(t,j)\in\dom\phi$, $t+j\ge T$ imply $j\ge \gamma_r(T)-N_r$, then $\mathcal{A}$ is uniformly globally asymptotically stable for the system \eqref{eq:hybrid}.
 \end{theorem}
The above stability result only requires the flow part of the system to be stabilizing while the jump part is only required to be non-expansive. This result is fully adapted to the analysis of LPV systems with piecewise differentiable parameters since discontinuous changes in the values of the parameters can not have any stabilizing effect -- quite the opposite. In this regard, the asymptotic stability of the system can only be ensured by the flow part of the system where parameters are smoothly varying.

\section{Main results}\label{sec:stab}

The objective of this section is to present the main results on the paper: the stability problem in the case of periodic parameter discontinuities is addressed in Section \ref{sec:cstDT} and extended to the aperiodic case in Section \ref{sec:minDT}. The results are connected to existing ones in Section \ref{sec:QR}. Finally, some computational discussions are provided in Section \ref{sec:computational}.

\subsection{Stability under constant dwell-time}\label{sec:cstDT}

In this section, we will consider the family of piecewise differentiable parameter trajectories given by
\begin{equation}
  \mathscr{P}_{\hspace{-2pt}\scriptscriptstyle{\bar{T}}}:=\left\{\begin{array}{c}
    \rho:\mathbb{R}_{\ge0}\to\mathcal{P}\left|\begin{array}{c}
      \dot{\rho}(t)\in\mathcal{Q}(\rho(t)),t\in(t_k,t_{k+1})\\
   T_k:=t_{k+1}-t_k=\bar T, k\in\mathbb{N}_{0}
    \end{array}\right.\end{array}\right\}
\end{equation}
where $\bar T>0$, $t_0=0$, and
$\mathcal{Q}(\rho)=\mathcal{Q}_1(\rho)\times\ldots\times\mathcal{Q}_N(\rho)$ with
\begin{equation}
  \mathcal{Q}_i(\rho):=\left\{\begin{array}{ccl}
    \mathcal{D}_i&&\textnormal{if }\rho_i\in(\underline{\rho}_i,\bar{\rho}_i),\\
    \mathcal{D}_i\cap\mathbb{R}_{\ge0}&&\textnormal{if }\rho_i=\underline{\rho}_i,\\
    \mathcal{D}_i\cap\mathbb{R}_{\le0}&&\textnormal{if }\rho_i=\bar{\rho}_i.
  \end{array}\right.
\end{equation}
In other words, the trajectories contained in this family can only exhibit jumps at the times $t_k=k\bar{T}$, $k>0$ (we assume that no discontinuity occurs at $t_0=0$) and hence the distance between two potential successive discontinuities -- the so-called \emph{dwell-time} -- is given by $T_k:=t_{k+1}-t_k=\bar T$ and is constant, whence the name \emph{constant dwell-time}. The associated  hybrid system is given by
\begin{equation}\label{eq:mainsystH_RDT}
\begin{array}{l}
  \left\{\left.\begin{array}{rcl}
    \dot{x}(t)&=&A(\rho(t))x(t)\\
    \dot{\rho}(t)&\in&\mathcal{Q}(\rho(t))\\
    \dot{\tau}(t)&=&1
  \end{array}\right| \textnormal{if }(x(t),\rho(t),\tau(t))\in C\right\}\\
\left\{\left.\begin{array}{rcl}
    x(t^+)&=&x(t)\\
    \rho(t^+)&\in&\mathcal{P}\\
    \tau(t^+)&=&0
  \end{array}\right|\textnormal{if }(x(t),\rho(t),\tau(t))\in D\right\}
  \end{array}
\end{equation}
where
\begin{equation}
        C := \mathbb{R}^n\times\mathcal{P}\times[0,\bar T)\ \textnormal{and}\ D :=\mathbb{R}^n\times\mathcal{P}\times\{\bar{T}\}.
\end{equation}
The initial condition for this system is chosen such that $(x(0),\rho(0),\tau(0))\in\mathbb{R}^n\times\mathcal{P}\times\{0\}$. Note, moreover, that $C\cap D=\emptyset$, $\bar{C}\cap D=D$ and $G(D)\subset C$. Hence, starting from the above initial condition, the solution of the hybrid system \eqref{eq:mainsystH_RDT} is complete (i.e. it is defined for all $t\ge0$) and it is confined in $C\cup D$.

The above hybrid system interestingly incorporates both the dynamics of the state of the system and the considered class of parameter trajectories. In addition to that, the state is also augmented to contain a clock that will measure the time elapsed since the last jump in the parameter trajectories. Indeed, starting from the chosen initial condition, the system will smoothly flow until the clock $\tau$ reaches the value $\bar T$ upon which the jump map is activated. The jump map changes the value of the parameter (thereby introducing a discontinuity in the parameter trajectories) and resets the clock, which places back the system in flow-mode. In this regard, we can easily observe that this formulation incorporates all the necessary information about the system and the parameter trajectories in order to provide accurate stability results. Such a result based on the use of a quadratic Lyapunov function is given below:
\begin{theorem}[Constant dwell-time]\label{th:cstDT}
Let $\bar T\in\mathbb{R}_{>0}$ be given and assume that there exist a bounded continuously differentiable matrix-valued function $S:[0, \bar{T}]\times\mathcal{P}\to\mathbb{S}^n_{\succ0}$ and a scalar $\eps>0$ such that the conditions
     \begin{equation}\label{eq:cst:1}
      \partial_\tau S(\tau,\theta)+\partial_\rho S(\tau,\theta)\mu+\He[S(\tau,\theta)A(\theta)]+\eps I_n\preceq0
    \end{equation}
    and
     \begin{equation}\label{eq:cst:2}
      S(0,\theta)-S(\bar{T},\eta)\preceq0
    \end{equation}
    hold for all $\theta,\eta\in\mathcal{P}$, all $\mu\in\mathcal{D}^v$ and all $\tau\in[0, \bar{T}]$.

    Then the LPV system \eqref{eq:mainsyst} with parameter trajectories in $\mathscr{P}_{\hspace{-2pt}\scriptscriptstyle{\bar{T}}}$ is asymptotically stable.\hfill\mendth
\end{theorem}
\begin{proof}
Let $\mathcal{A}=\{0\}\times\mathcal{P}\times[0, \bar{T}]$ and note that the LPV system \eqref{eq:mainsyst} with parameter trajectories in $\mathscr{P}_{\hspace{-2pt}\scriptscriptstyle{ T}}$ is asymptotically stable if and only if $\mathcal{A}$ is asymptotically stable for the system \eqref{eq:mainsystH_RDT}. We prove now that for the choice of the Lyapunov function  $V(x,\tau,\rho):=x^TS(\tau,\rho)x$ with $S(\tau,\rho)\succ0$ for all $\tau\in[0,T_{max}]$ and all $\rho\in\mathcal{P}$, the feasibility of the conditions of Theorem \ref{th:cstDT} implies the feasibility of those in Theorem \ref{th:goebel}.

Since $S(\tau,\rho)\succ0$ for all $\tau\in[0,T_{max}]$ and all $\rho\in\mathcal{P}$, then $V(x,\tau,\rho)=0$ if and only if $x=0$ and, hence, the conditions of statement (a) of Theorem \ref{th:goebel} are verified with the functions
\begin{equation}
\begin{array}{rcl}
\alpha_1(|(x,\tau,\rho)|_\mathcal{A})&=&\min_{(\tau,\rho)\in[0,\bar{T}]\times\mathcal{P}}\lambda_{min}(S(\tau,\rho))||x||_2^2\\
\alpha_2(|(x,\tau,\rho)|_\mathcal{A})&=&\max_{(\tau,\rho)\in[0,\bar{T}]\times\mathcal{P}}\lambda_{max}(S(\tau,\rho))||x||_2^2
\end{array}
\end{equation}
which are both $\mathcal{K}_\infty$ functions.

To prove that the feasibility of the condition \eqref{eq:cst:1} implies that of statement (b) of Theorem \ref{th:goebel}, let  $\Psi_{0}(\tau,\rho,\mu)$ be the matrix on the left-hand side of \eqref{eq:cst:1} when $\eps=0$. Using the linearity in $\mu$, it is immediate to get that $\Psi_0(\tau,\rho,\mu)+\eps I_n\preceq0$ for all $(\tau,\theta,\mu)\in[0,\bar{T}]\times\mathcal{P}\times\mathcal{D}^v$ if and only if $\Psi_0(\tau,\rho,\mu)+\eps I_n\preceq0$ for all $(\tau,\theta,\mu)\in[0,\bar{T}]\times\mathcal{P}\times\mathcal{D}$. Hence, \eqref{eq:cst:1} is equivalent to saying that
\begin{equation}
\langle \nabla V(x,\tau,\rho),f\rangle=x^T\Psi_0(\tau,\rho,\mu)x\le-\eps||x||_2^2
\end{equation}
for all $(x,\tau,\theta,\mu)\in\mathbb{R}^n\times[0,\bar{T}]\times\mathcal{P}\times\mathcal{D}$. Therefore, we have that the feasibility of the condition \eqref{eq:cst:1} implies that of statement (b) of Theorem \ref{th:goebel} with $\alpha_3(|(x,\tau,\rho)|_\mathcal{A})=\eps ||x||^2_2$, which is positive definite.

Finally, the condition of statement (c) reads
 \begin{equation}
V(g)-V(x)=x^T(S(0,\theta)-S(\bar{T},\eta))x\le0
\end{equation}
and must hold for all $(x,\theta,\eta)\in\mathbb{R}^n\times\mathcal{P}\times\mathcal{P}$. This is equivalent to the condition \eqref{eq:cst:2}.  Finally, we need to check the time-domain condition. First, note that
\begin{equation}
  \dom\phi=\bigcup_{j=0}^\infty([j\bar T,(j+1)\bar T],j)
\end{equation}
which, together  with $t+j\ge \tilde T$ for some $\tilde T>0$, implies that $t+1+t/\bar T\ge \tilde T$ since $j\le 1+t/\bar T$ for all $(t,j)\in\dom\phi$. Hence, $t\ge(1+\bar{T}^{-1})^{-1}(\tilde T-1)$ and, as a result, the set $\mathcal{A}$ is asymptotically stable for the system \eqref{eq:mainsystH_RDT} and the result follows.
\end{proof}

\subsection{Stability under minimum dwell-time}\label{sec:minDT}

Let us consider now the family of piecewise differentiable parameter trajectories given by
\begin{equation}
 \mathscr{P}_{\hspace{-1mm}{\scriptscriptstyle\geqslant\bar{T}}}:=\left\{\begin{array}{c}
    \rho:\mathbb{R}_{\ge0}\to\mathcal{P}\left|\begin{array}{c}
      \dot{\rho}(t)\in\mathcal{Q}(\rho(t)),  t\in[t_k,t_{k+1})\\
        T_k:=t_{k+1}-t_k\ge\bar T,\ k\in\mathbb{N}_{0}
    \end{array}\right.\end{array}\right\}
\end{equation}
where $\bar T>0$, $t_0=0$ (we again assume that no discontinuity can occur at $t_0$). The corresponding hybrid system is given by
\begin{equation}\label{eq:mainsystH_MDT}
\begin{array}{l}
  \left\{\left.\begin{array}{rcl}
    \dot{x}(t)&=&A(\rho(t))x(t)\\
    \dot{\rho}(t)&\in&\mathcal{Q}(\rho(t))\\
    \dot{\tau}(t)&=&1\\
    \dot{T}(t)&=&0\\
  \end{array}\right| \textnormal{if }(x(t),\rho(t),\tau(t),T(t))\in C\right\}\\
\left\{\left.\begin{array}{rcl}
    x(t^+)&=&x(t)\\
    \rho(t^+)&\in&\mathcal{P}\\
    \tau(t^+)&=&0\\
   T(t^+)&\in&[\bar{T},\infty)
  \end{array}\right|\textnormal{if }(x(t),\rho(t),\tau(t),T(t))\in D\right\}
  \end{array}
\end{equation}
where
\begin{equation}
  \begin{array}{rcl}
        C&=& \mathbb{R}^n\times\mathcal{P}\times E_<,\\
        D&=&\mathbb{R}^n\times\mathcal{P}\times E_=\\
        E_\square&=&\{\varphi\in\mathbb{R}_{\ge0}\times[\bar T,\infty):\varphi_1\square\varphi_2\},\ \square\in\{<,=\}.
  \end{array}
\end{equation}
The initial condition for this system is chosen such that $(x(0),\rho(0),\tau(0),T(0))\in\mathbb{R}^n\times\mathcal{P}\times\{0\}\times [\bar T,\infty)$. This system contains an additional state compared to the system considered in Section \ref{sec:cstDT} in order to avoid the use of a time-dependent jump set (a purely technical requirement that can be relaxed). Using the current formulation, the current dwell-time is drawn each time the system jumps and the jumping condition is satisfied when $\tau(t)=T(t)$. This leads to the following result:
\begin{theorem}[Minimum dwell-time]\label{th:minDT}
Let $\bar T\in\mathbb{R}_{>0}$ be given and assume that there exist a bounded continuously differentiable matrix-valued function $S:[0, \bar{T}]\times\mathcal{P}\to\mathbb{S}^n_{\succ0}$ and a scalar $\eps>0$ such that the conditions
     \begin{equation}\label{eq:minDT:1}
      \partial_\rho S(\bar{T},\theta)\mu+\He[S(\bar{T},\theta)A(\theta)]+\eps I\preceq0
    \end{equation}
       \begin{equation}\label{eq:minDT:2}
      \partial_\tau S(\tau,\theta)+\partial_\rho S(\tau,\theta)\mu+\He[S(\tau,\theta)A(\theta)]+\eps I\preceq0
    \end{equation}
    and
     \begin{equation}\label{eq:minDT:3}
      S(0,\theta)-S(\bar{T},\eta)\preceq0
    \end{equation}
    hold for all $\theta,\eta\in\mathcal{P}$, $\mu\in\mathcal{D}^v $ and all $\tau\in[0, \bar{T}]$. Then, the LPV system \eqref{eq:mainsyst} with parameter trajectories in $\mathscr{P}_{\hspace{-1mm}{\scriptscriptstyle\geqslant \bar{T}}}$ is asymptotically stable.
\end{theorem}

\begin{proof}
Assume that the full trajectory of $T(t)$ is known and such that $T(t)\ge\bar T$ for all $t\ge0$. Note that this is possible since $T(t)$ is independent of the rest of the state of the hybrid system \eqref{eq:mainsystH_MDT}. Then, there exists a $T_{max}<\infty$ such that $\bar{T}\le T(t)\le T_{max}$ for all  $t\ge0$ and define the set $\mathcal{A}=\{0\}\times\mathcal{P}\times((E_<\cup E_=)\cap[0,T_{max}]^2)$. Let us consider here the following Lyapunov function
  \begin{equation}
    V(x,\tau,\rho)=\left\{\begin{array}{lcl}
      x^TS(\tau,\rho)x&&\textnormal{if }\tau\le\bar{T},\\
      x^TS(\bar{T},\rho)x&&\textnormal{if }\tau>\bar{T}.
    \end{array}\right.
  \end{equation}
  The rest of the proof is analogous to that of Theorem \ref{th:cstDT} and it thus omitted.
\end{proof}

\subsection{Connection with quadratic and robust stability}\label{sec:QR}

Interestingly, it can be shown that the minimum dwell-time result stated in Theorem \ref{th:minDT} naturally generalizes and unifies the quadratic and robust stability conditions through the concept of minimum dwell-time. This is further explained in the results below:
\begin{proposition}[Quadratic stability]
  When $\bar T\to0$ and $t_k\to\infty$ as $k\to\infty$, then the conditions of Theorem \ref{th:minDT} are equivalent to saying that there exists a matrix $P\in\mathbb{S}_{\succ0}^n$ such that
  \begin{equation}\label{eq:quadstab}
        A(\theta)^TP+PA(\theta)\prec0
  \end{equation}
  for all $\theta\in\mathcal{P}$.
\end{proposition}
\begin{proof}
First note that since $\bar T>0$ and $t_k\to\infty$ as $k\to\infty$, then the solution of the hybrid system is complete. Let $\bar T=\epsilon>0$, then we have that
\begin{equation}
  \begin{array}{rcl}
    S(0,\theta)-S(\epsilon,\theta)&=&-\epsilon \partial_\tau S(0,\theta)+o(\epsilon),\\
    S(0,\theta)-S(\epsilon,\eta)&=&(S(0,\theta)-S(0,\eta))-\epsilon \partial_\tau S(0,\eta)+o(\epsilon),\\
    S(0,\eta)-S(\epsilon,\theta)&=&(S(0,\eta)-S(0,\theta))-\epsilon \partial_\tau S(0,\theta)+o(\epsilon)
  \end{array}
\end{equation}
where it is assumed that $\eta\ne\theta$ and where $o(\epsilon)$ is the little-o notation. From \eqref{eq:minDT:3}, we get that the first equation implies that $\partial_\tau S(0,\theta)\succeq0$. Therefore, for the second expression to be negative semidefinite for any arbitrarily small $\epsilon>0$, we need that $S(0,\theta)-S(0,\eta)$ be negative semidefinite. However, this contradicts the last one and hence we need that $S(0,\eta)=S(0,\theta)$; i.e. $S$ is independent of $\rho$. Finally, since $\epsilon>0$ is arbitrarily small, then both \eqref{eq:minDT:1}-\eqref{eq:minDT:2} can be satisfied with a matrix-valued function $S$ that is independent of $\tau$. Hence, we need that $S(\tau,\theta):=P\succ0$ and substituting it in  \eqref{eq:minDT:1}-\eqref{eq:minDT:2} yield the quadratic stability condition \eqref{eq:quadstab}.
\end{proof}

\begin{proposition}[Robust stability]
  When $\bar T\to\infty$, then the conditions of Theorem \ref{th:minDT} are equivalent to saying that there exists a differentiable matrix-valued function $P:\mathcal{P}\mapsto\mathcal{S}_{\succ0}^n$ such that
  \begin{equation}\label{eq:robstab}
      \partial_\rho P(\theta)\mu+A(\theta)^TP(\theta)+P(\theta)A(\theta)\prec0
  \end{equation}
  for all $\theta\in\mathcal{P}$ and all $\mu\in\mathcal{D}^v $.
\end{proposition}
\begin{proof}
Clearly, when $\bar T\to\infty$, then there are no jumps anymore and hence we can remove the condition \eqref{eq:minDT:3} as it never occurs. Consequently, we can choose $S(\tau,\theta)=P(\theta)$ and the conditions \eqref{eq:minDT:1}-\eqref{eq:minDT:2}  immediately reduce to \eqref{eq:robstab}.
\end{proof}

\subsection{Computational considerations}\label{sec:computational}

The conditions formulated in Theorem \ref{th:cstDT} and Theorem \ref{th:minDT} are infinite-dimensional semidefinite programs that are intractable per se. To make them tractable, we propose to consider an approach based on sum of squares programming \cite{Parrilo:00} that will result in a approximate finite-dimensional semidefinite program that can be solved using standard semidefinite programming solvers such as SeDuMi \cite{Sturm:01a}. The relaxation procedure can be performed with the help of the package SOSTOOLS \cite{sostools3} This conversion to an SOS program is described below.

Assuming that $\mathcal{P}$ is a compact semialgebraic set, then it can defined as
\begin{equation}
  \mathcal{P}=\left\{\theta\in\mathbb{R}^N: g_i(\theta)\ge0, i=1,\ldots,M \right\}
\end{equation}
where the functions $g_i:\mathbb{R}^N\mapsto\mathbb{R}$, $i=1,\ldots,M $ are polynomial. We further have that
\begin{equation}
  [0, \bar{T}]=\left\{\tau\in\mathbb{R}:\ f(\tau):=\tau( \bar{T}-\tau)\ge0\right\}.
\end{equation}
In what follows, we say that a symmetric matrix-valued function $\Theta(\cdot)$ is a sum of squares matrix (SOS matrix) if there exists a matrix $\Xi(\cdot)$ such that $\Theta(\cdot)=\Xi(\cdot)^{T}\Xi(\cdot)$. The following result provides the sum of squares formulation of Theorem \ref{th:cstDT}:
\begin{proposition}\label{prog:periodic}
  Let $\eps, \bar{T}>0$ be given and  assume that the sum of squares program

{\vspace{4mm}}
\noindent\fbox{
\parbox{\textwidth}{
    Find polynomial matrices
      \begin{equation*}
        \begin{array}{rcl}
    S,\Gamma_1,\Gamma_2,\Gamma_1^i,\Gamma_2^i&:&[0,\bar{T}]\times\mathcal{P}\mapsto\mathbb{S}^n,\\
    \Gamma_3^i,\Gamma_4^i&:&\mathcal{P}\times\mathcal{P}\mapsto\mathbb{S}^n,\\
    \Upsilon_1^i,\Upsilon_2^i&:&\mathcal{D}^v \times [0,\bar{T}]\times\mathcal{P}\mapsto\mathbb{S}^n
    \end{array}
      \end{equation*}
      where $i=1,\ldots,M$ and such that
      \begin{itemize}
        \item $\Gamma_1,\Gamma_2,\Gamma_1^i,\Gamma_2^i,\Gamma_3^i,\Gamma_4^i,\Upsilon_1^i,\Upsilon_2^i$, $i=1,\ldots,M$, are SOS matrices for all $\mu\in\mathcal{D}^v$,
        \item $S(\tau,\theta)-\sum_{i=1}^{M }\Gamma_1^i(\theta)g_i(\theta)-\eps I_n$ is an SOS matrix
        \item $-\partial_\rho S(\tau,\theta)\mu-\partial_\tau S(\tau,\theta)-\He[S(\tau,\theta)A(\theta)]-\sum_{i=1}^{M }\Upsilon_1^i(\tau,\theta,\mu)(\tau,\theta)g_i(\theta)$

        $\qquad\qquad\qquad\qquad\qquad\qquad\qquad-\Upsilon_2(\tau,\theta,\mu)\tau( T-\tau)$ is an SOS matrix for all $\mu\in\mathcal{D}^v $
        \item $ S( T,\eta)-S(0,\theta)-\eps I_n-\sum_{i=1}^{M }\Gamma_3^i(\theta,\eta)g_i(\theta) -\sum_{i=1}^{M }\Gamma_4^i(\theta,\eta)g_i(\eta)$ is an SOS matrix.
    \end{itemize}}}

is feasible. Then, the conditions of Theorem \ref{th:cstDT} hold with the computed polynomial matrix $S(\tau,\theta)$ and the system \eqref{eq:mainsyst} is asymptotically stable for all $\rho\in\mathscr{P}_{\hspace{-2pt}\scriptscriptstyle{\bar T}}$.
\end{proposition}

\begin{remark}
The conditions of Theorem \ref{th:minDT} can be checked by simply adding  to the SOS program of Proposition \ref{prog:periodic} the constraints
\begin{itemize}
  \item $ \Upsilon_3^i(\tau,\theta,\mu)$, $i=1,\ldots,M$,  are SOS matrices for all $\mu\in\mathcal{D}^v $
  \item $ -\partial_\rho S(\bar{T},\theta)\mu-\He[S( T,\theta)A(\theta)]-\sum_{i=1}^{M }\Upsilon_3^i(\tau,\theta,\mu)g_i(\theta)-\eps I_n$  is an SOS matrix for all $\mu\in\mathcal{D}^v $
\end{itemize}
where $\Upsilon_3^i:\mathcal{D}^v \times [0,\bar{T}]\times\mathcal{P}\mapsto\mathbb{S}^n$, $i=1,\ldots,M$, are additional polynomial variables.
\end{remark}

\begin{remark}
  When the parameter set $\mathcal{P}$ is also defined by equality constraints $h_i(\theta)=0$, $i=1,\ldots, M'$, these constraints can be simply added in the sum of squares programs in the same way as the inequality constraints, but with the particularity that the corresponding multiplier matrices be simply symmetric instead of being symmetric SOS.
\end{remark}


\section{Examples}\label{sec:examples}

We consider now two examples. The first one is a 2-dimensional toy example considered in \cite{Xie:97} whereas the second one is a 4-dimensional system considered in \cite{Wu:95} and inspired from an automatic flight control design problem. 

\subsection{Example 1}

Let us consider the system \eqref{eq:mainsyst} with the matrix  \cite{Xie:97,Briat:15d}
\begin{equation}\label{eq:ex1}
  A(\rho)=\begin{bmatrix}
    0 & &1\\
    -2-\rho &\ &-1
  \end{bmatrix}
\end{equation}
where the time-varying parameter $\rho(t)$ takes values in $\mathcal{P}=[0,\bar{\rho}]$, $\bar{\rho}>0$. It is known \cite{Xie:97} that this system is quadratically stable if and only if $\bar{\rho}\le3.828$ but it is was later proven in the context of piecewise constant parameters \cite{Briat:15d} that this bound can be improved provided that discontinuities do not occur too often. We now apply the conditions of Theorem \ref{th:cstDT} and Theorem \ref{th:minDT} in order to characterize the impact of parameter variations between discontinuities. To this aim, we consider that $|\dot{\rho}(t)|\le\nu$ with $\nu\ge0$ and that $\bar{\rho}\in\{0,0.1,\ldots,10\}$. For each value for the upper-bound $\bar{\rho}$ in that set, we solve for the conditions Theorem \ref{th:cstDT} and Theorem \ref{th:minDT}  to get estimates (i.e. upper-bounds) for the minimum stability-preserving constant and minimum dwell-time. We use here $\eps_1=\eps_2=\eps_3=0.01$ and polynomials of degree 4 in the sum of squares programs. Note that we have, in this case, $M=1$, $M'=0$ and $g_1(\theta)=\theta(\bar{\rho}-\theta)$. The complexity of the approach can be evaluated here through the number of primal/dual variables of the semidefinite program which is 2209/273 in the constant dwell-time case and 2409/315 in the minimum dwell-time case. Time-wise, the average preprocessing/solving time is given by 4.83/0.89 sec in the constant dwell-time case and 6.04/1.25 sec in the minimum dwell-time case. The results are depicted in Fig.~\ref{fig:1} and Fig.~\ref{fig:2} where we can see that the obtained minimum values for the dwell-times increase with the rate of variation $\nu$ of the parameter, which is an indicator of the fact that increasing the rate of variation of the parameter destabilizes the system and, consequently, the dwell-time needs to increase in order to preserve the overall stability of the system. For this example, it is interesting to note that the constant and minimum dwell-time curves seem to converge to each other when $\nu$ increases which could indicate that when the parameter is fast-varying between jumps then the aperiodicity of jumps do not affect much the stability of the system.

\begin{figure}
  \centering
  \includegraphics[width=0.65\textwidth]{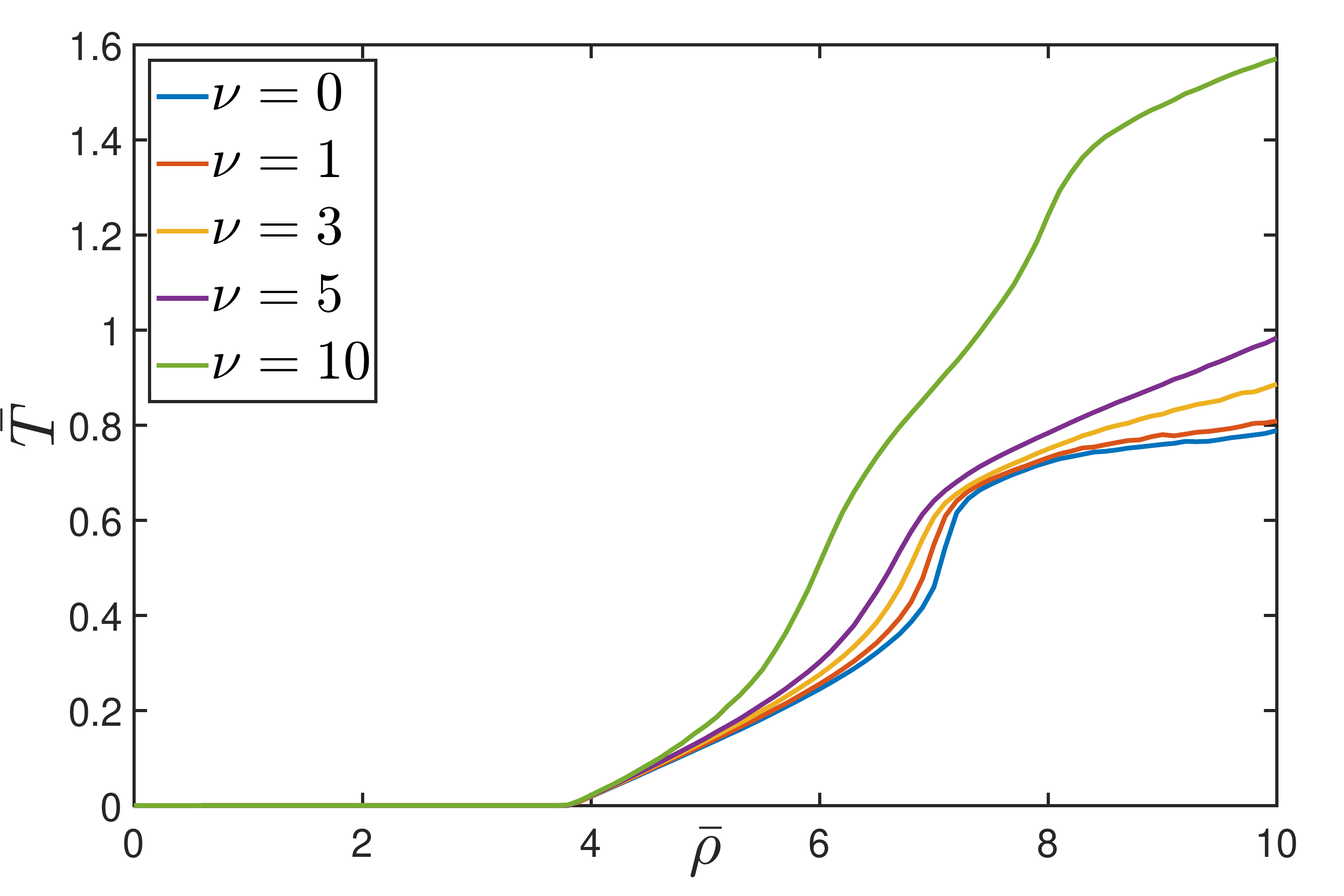}
  \caption{Evolution of the computed minimum upper-bound on the constant dwell-time using Theorem \ref{th:cstDT} for the system \eqref{eq:mainsyst}-\eqref{eq:ex1} using an SOS approach with polynomials of degree 4.}\label{fig:1}
\end{figure}

\begin{figure}
  \centering
  \includegraphics[width=0.65\textwidth]{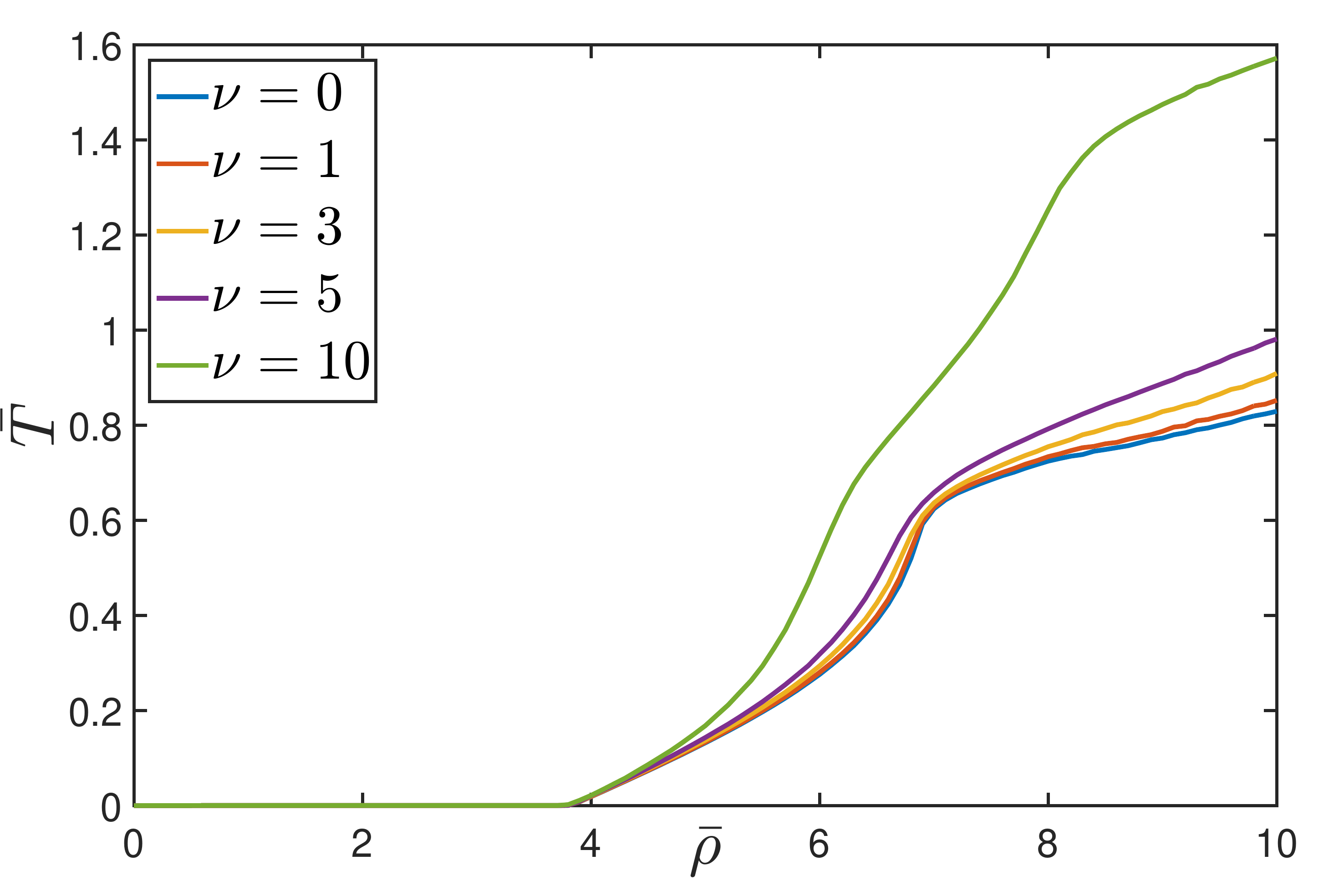}
  \caption{Evolution of the computed minimum upper-bound on the minimum stability-preserving minimum dwell-time using Theorem \ref{th:minDT} for the system \eqref{eq:mainsyst}-\eqref{eq:ex1} with $|\dot{\rho}|\le\nu$ using an SOS approach with polynomials of degree 4.}\label{fig:2}
\end{figure}
\subsection{Example 2}

Let us consider now the system \eqref{eq:mainsyst} with the matrix considered in \cite[p. 55]{Wu:95}:
  \begin{equation}\label{eq:ex2}
  A(\rho)=\begin{bmatrix}
   3/4 &\ & 2 &\ & \rho_1 &\ & \rho_2\\
    0 &\ & 1/2 &\ & -\rho_2 &\ & \rho_1\\
  -3\upsilon\rho_1/4   &\ & \upsilon\left(\rho_2-2\rho_1\right) &\ & -\upsilon &\ & 0\\
  -3\upsilon\rho_2/4    &\ & \upsilon\left(\rho_1-2\rho_2\right)  &\ & 0 &\ & -\upsilon
      \end{bmatrix}
\end{equation}
where $\upsilon=15/4$ and $\rho\in\mathcal{P}=\{z\in\mathbb{R}^2:||z||_2=1\}$. It has been shown in \cite{Wu:95} that this system is not quadratically stable but was proven to be stable under minimum dwell-time equal to 1.7605 when the parameter trajectories are piecewise constant \cite{Briat:15d}. We propose now to quantify the effects of smooth parameter variations between discontinuities. Note, however, that the set $\mathcal{P}$ is not a box as considered along the paper but the next calculations show that this is not a problem since a proper set for the values of the derivative of the parameters can be defined. To this aim, let us define the parametrization $\rho_1(t)=\cos(\beta(t))$ and $\rho_2(t)=\sin(\beta(t))$ where $\beta(t)$ is piecewise differentiable. Differentiating these equalities yields $\dot{\rho}_1(t)=-\dot{\beta}(t)\rho_2(t)$ and $\dot{\rho}_2(t)=\dot{\beta}(t)\rho_1(t)$ where $\dot{\beta}(t)\in[-\nu,\nu]$, $\nu\ge0$, at all times where $\beta(t)$ is differentiable. In this regard, we can consider $\dot{\beta}$ as an additional parameter that enters linearly in the stability conditions and hence the conditions can be checked at the vertices of the interval, that is, for all $\dot{\beta}\in\{-\nu,\nu\}$. Note that, in this case, we have $M=0$, $M'=1$ and $h_1(\rho)=\rho_1^2+\rho_2^2-1$.

\begin{table*}
  \centering
    \caption{Evolution of the computed minimum upper-bound on the minimum dwell-time using Theorem \ref{th:minDT} for the system \eqref{eq:mainsyst}-\eqref{eq:ex2} with $|\dot{\beta}|\le\nu$ using an SOS approach with polynomials of degree $d$. The number of primal/dual variables of the semidefinite program and the preprocessing/solving time are also given.}\label{tab}

 {\footnotesize \begin{tabular}{|c||c|c|c|c|c||c|c|c|}
  \hline
   & $\nu=0$ & $\nu=0.1$ & $\nu=0.3$ & $\nu=0.5$ & $\nu=0.8$ & $\nu=0.9$ & primal/dual vars. & time (sec)\\
  \hline
  \hline
    $d=2$  & 2.7282 & 2.9494 & 3.5578 & 4.6317 & 11.6859 & 26.1883 & 9820/1850 & 20/27\\
    $d=4$  & 1.7605 & 1.8881 & 2.2561 & 2.9466 & 6.4539 & \red{num. err.} & 43300/4620 & 212/935\\
    \hline
  \end{tabular}}
\end{table*}
We now consider the conditions of Theorem \ref{th:minDT} and we get the results gathered in Table \ref{tab} where we can see that, as expected, when $\nu$ increases then the minimum dwell-time has to increase to preserve stability. Using polynomials of higher degree allows to improve the numerical results at the expense of an increase of the computational complexity. As a final comment, it seems important to point out the failure of the semidefinite solver due to too important numerical errors when $d=4$ and $\nu=0.9$.

\section{Conclusion}

Reformulating LPV systems into hybrid systems enabled the derivation of tractable conditions for establishing the stability of LPV systems with piecewise differentiable parameters. These results extend those obtained in \cite{Briat:15d} for piecewise constant parameters through the use of a different, yet connected, approach. It is shown that the obtained stability conditions generalize and unify the well-known quadratic and robust stability conditions in a single formulation. Possible extensions include controller/filter/observer design and performance analysis.

%
%
%


\begin{thebibliography}{10}

\bibitem{Toth:10}
R.~T{\'{o}th}.
\newblock {\em Modeling and Identification of Linear Parameter-Varying
  Systems}.
\newblock Springer, Germany, 2010.

\bibitem{Mohammadpour:12}
J.~Mohammadpour and C.~W. Scherer, editors.
\newblock {\em Control of Linear Parameter Varying Systems with Applications}.
\newblock Springer, New York, USA, 2012.

\bibitem{Briat:book1}
C.~Briat.
\newblock {\em Linear Parameter-Varying and Time-Delay Systems -- Analysis,
  Observation, Filtering \& Control}, volume~3 of {\em Advances on Delays and
  Dynamics}.
\newblock Springer-Verlag, Heidelberg, Germany, 2015.

\bibitem{Wu:95}
F.~Wu.
\newblock {\em Control of linear parameter varying systems}.
\newblock PhD thesis, University of California Berkeley, 1995.

\bibitem{Shamma:88phd}
J.~S. Shamma.
\newblock {\em Analysis and design of gain-scheduled control systems}.
\newblock PhD thesis, Laboratory for Information and decision systems -
  Massachusetts Institute of Technology, 1988.

\bibitem{Shamma:92}
J.~S. Shamma and M.~Athans.
\newblock Gain scheduling: potential hazards and possible remedies.
\newblock {\em {IEEE} Contr. Syst. Magazine}, 12(3):101--107, 1992.

\bibitem{Poussot:08b}
C.~Poussot-Vassal, O.~Sename, L.~Dugard, P.~G\'asp\'ar, Z.~Szab\'o, and
  J.~Bokor.
\newblock New semi-active suspension control strategy through {LPV} technique.
\newblock {\em Control Engineering Practice}, 2008.

\bibitem{Poussot:10}
S.~M. Savaresi, C.~{Poussot-Vassal}, C.~Spelta, O.~Sename, and L.~Dugard.
\newblock {\em Semi-Active Suspension Control Design for Vehicles}.
\newblock Butterworth Heinemann, 2010.

\bibitem{KajiwaraAG:99a}
H.~Kajiwara, P.~Apkarian, and P.~Gahinet.
\newblock {LPV} techniques for control of an inverted pendulum.
\newblock {\em IEEE Control System Magazine}, 19:44--54, 1999.

\bibitem{Robert:10}
D.~Robert, O.~Sename, and D.~Simon.
\newblock An $\mathcal{H}_\infty$ {LPV} design for sampling varying
  controllers: Experimentation with a {T}-inverted pendulum.
\newblock {\em {IEEE} Transactions on Control Systems Technology},
  18(3):741--749, 2010.

\bibitem{Gilbert:10}
W.~Gilbert, D.~Henrion, J.~Bernussou, and D.~Boyer.
\newblock Polynomial {LPV} synthesis applied to turbofan engines.
\newblock {\em Control Engineering Practice}, 18(9):1077--1083, 2010.

\bibitem{Seiler:12}
P.~Seiler, G.~J. Balas, and A.~Packard.
\newblock Linear parameter-varying control for the {X}-53 active aeroelastic
  wing.
\newblock In J.~Mohammadpour and C.~W. Scherer, editors, {\em Control of Linear
  Parameter Varying Systems with Applications}, pages 483--512. Springer
  Science, 2012.

\bibitem{Pfifer:15}
H.~Pfifer, C.~P. Moreno, J.~Theis, A.~Kotikapuldi, A.~Gupta, B.~Takarics, and
  P.~Seiler.
\newblock Linear parameter varying techniques applied to aeroservoelastic
  aircraft: In memory of {G}ary {B}alas.
\newblock In {\em 1st IFAC Workshop on Linear Parameter Varying Systems}, 2015.

\bibitem{Packard:94a}
A.~Packard.
\newblock Gain scheduling via {L}inear {F}ractional {T}ransformations.
\newblock {\em Systems \& Control Letters}, 22:79--92, 1994.

\bibitem{Apkarian:95}
P.~Apkarian, P.~Gahinet, and G.~Becker.
\newblock Self-scheduled control of linear parameter varying systems: A design
  example.
\newblock {\em Automatica}, 31(9):1251--1261, 1995.

\bibitem{Apkarian:95a}
P.~Apkarian and P.~Gahinet.
\newblock A convex characterization of gain-scheduled $\mathcal{H}_\infty$
  controllers.
\newblock {\em IEEE Transactions on Automatic Control}, 5:853--864, 1995.

\bibitem{Apkarian:98a}
P.~Apkarian and R.~J. Adams.
\newblock Advanced gain-scheduling techniques for uncertain systems.
\newblock {\em IEEE Transactions on Control Systems Technology}, 6:21--32,
  1998.

\bibitem{Wu:01}
F.~Wu.
\newblock A generalized {LPV} system analysis and control synthesis framework.
\newblock {\em International Journal of Control}, 74:745--759, 2001.

\bibitem{Wu:06b}
F.~Wu and K.~Dong.
\newblock Gain-scheduling control of {LFT} systems using parameter-dependent
  {L}yapunov functions.
\newblock {\em Automatica}, 42:39--50, 2006.

\bibitem{Scherer:01}
C.~W. Scherer.
\newblock {LPV} control and full-block multipliers.
\newblock {\em Automatica}, 37:361--375, 2001.

\bibitem{Bokor:05a}
J.~Bokor.
\newblock Linear parameter varying systems : a geometric theory and
  applications.
\newblock {\em IFAC 16th World Congrees, Prague}, 2005.

\bibitem{Scherer:12}
C.~W. Scherer and I.~E. K{\"{o}}se.
\newblock {Gain-Scheduled Control Synthesis using Dynamic {D}-Scales}.
\newblock {\em {IEEE Transactions on Automatic Control}}, 57(9):2219--2234,
  2012.

\bibitem{Scherer:15}
C.~W. Scherer.
\newblock Gain-scheduling control with dynamic multipliers by convex
  optimization.
\newblock {\em SIAM Journal on Control and Optimization}, 53(3):1224--1249,
  2015.

\bibitem{Briat:15d}
C.~Briat.
\newblock Stability analysis and control of {LPV} systems with piecewise
  constant parameters.
\newblock {\em Systems \& Control Letters}, 82:10--17, 2015.

\bibitem{Balas:15}
G.~Balas, A.~Hjartarson, A.~Packard, and P.~Seiler.
\newblock {LPVTools: A Toolbox for Modeling, Analysis, and Synthesis of
  Parameter Varying Control Systems}.
\newblock Software and user's manual, 2015.

\bibitem{Peni:16}
T.~Peni and P.~J. Seiler.
\newblock Computation of lower bounds for the induced $\mathcal{{L}}_2$ norm of
  {LPV} systems.
\newblock {\em International Journal of Robust and Nonlinear Control},
  26(4):646--661, 2016.

\bibitem{Wang:16}
S.~Wang, H.~Pfifer, and P.~Seiler.
\newblock Robust synthesis for linear parameter varying systems using integral
  quadratic constraints.
\newblock {\em Automatica}, 68:111--118, 2016.

\bibitem{Goebel:12}
R.~Goebel, R.~G. Sanfelice, and A.~R. Teel.
\newblock {\em Hybrid Dynamical Systems. Modeling, Stability, and Robustness}.
\newblock Princeton University Press, 2012.

\bibitem{Briat:13d}
C.~Briat.
\newblock Convex conditions for robust stability analysis and stabilization of
  linear aperiodic impulsive and sampled-data systems under dwell-time
  constraints.
\newblock {\em Automatica}, 49(11):3449--3457, 2013.

\bibitem{Briat:14f}
C.~Briat.
\newblock Convex conditions for robust stabilization of uncertain switched
  systems with guaranteed minimum and mode-dependent dwell-time.
\newblock {\em Systems \& Control Letters}, 78:63--72, 2015.

\bibitem{Briat:15i}
C.~Briat.
\newblock Stability analysis and stabilization of stochastic linear impulsive,
  switched and sampled-data systems under dwell-time constraints.
\newblock {\em Automatica}, 74:279--287, 2016.

\bibitem{Putinar:93}
M.~Putinar.
\newblock Positive polynomials on compact semi-algebraic sets.
\newblock {\em Indiana Univ. Math. J.}, 42(3):969--984, 1993.

\bibitem{Parrilo:00}
P.~Parrilo.
\newblock {\em Structured Semidefinite Programs and Semialgebraic Geometry
  Methods in Robustness and Optimization}.
\newblock PhD thesis, California Institute of Technology, Pasadena, California,
  2000.

\bibitem{sostools3}
A.~Papachristodoulou, J.~Anderson, G.~Valmorbida, S.~Prajna, P.~Seiler, and
  P.~A. Parrilo.
\newblock {\em {SOSTOOLS}: Sum of squares optimization toolbox for {MATLAB}
  v3.00}, 2013.

\bibitem{Sturm:01a}
J.~F. Sturm.
\newblock Using {SEDUMI} $1. 02$, a {M}atlab {T}oolbox for {O}ptimization
  {O}ver {S}ymmetric {C}ones.
\newblock {\em Optimization Methods and Software}, 11(12):625--653, 2001.

\bibitem{Xie:97}
L.~Xie, S.~Shishkin, and M.~Fu.
\newblock Piecewise {L}yapunov functions for robust stability of linear
  time-varying systems.
\newblock {\em Systems \& Control Letters}, 31(3):165--171, 1997.

\end{thebibliography}

\end{document}